\documentclass[11pt,leqno]{amsart}
\usepackage{amssymb,amsmath,amsthm,dsfont,fourier}

\usepackage{mathtools}
\usepackage{array}
\usepackage[text={6.5in,8.5in},centering]{geometry}
\usepackage[dvipsnames]{xcolor}
\usepackage{enumitem}
\usepackage{tikz-cd}
\usetikzlibrary{decorations.pathmorphing}
\usepackage[colorlinks=true,citecolor=NavyBlue,linkcolor=NavyBlue,urlcolor=NavyBlue]{hyperref}

\DeclareMathOperator{\m}{M}

\newcommand{\cT}{\mathcal{T}}

\newtheorem{theorem}{Theorem}[section]
\newtheorem{proposition}[theorem]{Proposition}
\newtheorem{lemma}[theorem]{Lemma}
\newtheorem{corollary}[theorem]{Corollary}
\theoremstyle{definition}
\newtheorem{definition}{Definition}[section]
\theoremstyle{remark}
\newtheorem{remark}[theorem]{Remark}
\newtheorem{cremark}[theorem]{Concluding Remarks}

\numberwithin{equation}{section}

\numberwithin{equation}{section}

\DeclareMathOperator{\bF}{{\mathds F}}

\DeclareMathOperator{\bN}{{\mathds N}}

\begin{document}
	
\title{$\m$-ideals: from Banach spaces to rings}
	
\author{David P. Blecher}
	
\address{Department of Mathematics, University of Houston, Houston, TX 77204-3008.}
	
\email{dpbleche@central.uh.edu}
	
\author{Amartya Goswami}

\address{[1]  Department of Mathematics and Applied Mathematics, University of Johannesburg, P.O. Box 524, Auckland Park 2006, South Africa. [2]  National Institute for Theoretical and Computational Sciences (NITheCS), South Africa.}
	
\email{agoswami@uj.ac.za}
	

	
\begin{abstract}
We introduce and investigate a class of ring ideals, termed ring \emph{$\m$-ideals}, inspired by the Alfsen--Effros theory of $\m$-ideals in Banach spaces.  We show that $\m$-ideals extend the classical notion of essential ideals and subsume them as a subclass. The central theorem provides a full characterization: an ideal is an $\m$-ideal if and only if it is either essential or relatively irreducible. This dichotomy reveals the abundant and diverse nature of $\m$-ideals, encompassing both essential and minimal ideals, and admits natural generalizations in rings beyond the commutative and unital settings.

We systematically study the algebraic stability of $\m$-ideals under standard constructions such as intersection, quotients by direct-summand ideals, direct product, and Morita equivalence and establish their behavior in various subclasses of
rings,
such as operator algebras. In certain rings such as $\mathds{Z}_n$ and C*-algebras, we completely classify  ring $\m$-ideals and relate them to algebraically minimal projections and central idempotents. The ring $\m$-ideals in $C(K)$ are shown to be precisely the essential ideals or those minimal ideals corresponding to isolated points.

Structurally, for unital rings we show that the absence of proper $\m$-ideals characterizes simplicity, while rings in which every proper $\m$-ideal is a direct summand must decompose as finite direct sums of simple rings. In closing, we introduce the notion of $\m$-complements, drawing an analogy with essential extensions in module theory, and demonstrate their existence.\end{abstract}
\maketitle

\tableofcontents	
\section{Introduction}
The $\m$-ideal  theory  in a Banach space  is one of the most important tools in the isometric theory of Banach spaces and 
constitutes a formidable technology \cite{HWW93}. 
Drawing an analogy from their beautiful characterization by Alfsen and Effros, we are led to define a class of ideals in rings.   It turns out that this class contains the essential ideals.  
Recall that a proper ideal $J$ of a ring $R$ is called \emph{essential} if $J\cap I\neq 0$ for all nonzero ideals $I$ of $R$. As a generalization thereof, essential submodules play a pivotal role in the theory of injective modules via the notion of injective hulls. Borrowing ideas from various concepts related to essential modules, we introduce an analogous notion in the context of $\m$-ideals.

The aim of this paper is twofold: on the one hand, to systematically develop a theory of ring $\m$-ideals as a generalization of essential ideals; on the other hand, to relate this class of ideals to the corresponding notion in Banach spaces and C*-algebras.  The subtitle of our paper, in some sense, should be ``from Banach spaces back to rings" because part of
Alfsen and Effros' motivation was to `generalize' to Banach spaces the usual `calculus' of ideals in rings \cite{Lam01,Row88}.  By now, however, $\m$-ideals in Banach spaces have a beautiful and formidable theory of their own \cite{HWW93}, and some of these ideas may in turn be applied in algebra.

We now describe the structure of our paper.
We investigate the stability of this class under the usual constructions with ring ideals, such as intersections and quotients by direct-summand ideals. 
We establish that $\m$-ideals strictly generalize essential ideals  (Proposition \ref{eim}), and provide a complete characterization: an ideal is an $\m$-ideal if and only if it is essential or relatively irreducible (Theorem \ref{chmi}). This characterization allows us to classify $\m$-ideals explicitly in important classes of rings, such as $ \mathds{Z}_n$ (Theorem \ref{charzn}), $ C(K)$ (Theorem \ref{ck}), and unital C*-algebras.  Indeed, the nonzero ring $\m$-ideals in a C*-algebra  are precisely either the essential ideals or the  ideals whose closure is a prime C*-algebra 
(See Corollary \ref{nck}.  Prime C*-algebras are well understood, and have a large theory, e.g.\ \cite{Bla06}.)   We show that $\m$-ideals are not preserved under direct sums or tensor products. We show that any unital ring with a proper ideal necessarily has a proper $\m$-ideal (Corollary \ref{nmi} and the lemma before it). Further, we show that Morita equivalence preserves M-ideals (Theorem \ref{morita}) and that every ideal of a unital ring is contained in a maximal $\m$-ideal unless it is a direct summand. From a structural standpoint, we prove  for  unital rings that the absence of proper $\m$-ideals characterizes simple rings and that rings in which every proper $\m$-ideal is a direct summand must decompose as finite direct sums of simple rings (Theorem \ref{dss}). Motivated by analogies with essential extensions in module theory, we introduce the notion of $\m$-complements and establish their universal existence (Theorem \ref{mct} and its corollary), thereby enriching the interplay between ideal theory and ring decomposition.

\section{Basic Properties}

According to Theorem 5.8 in \cite{AE72}, a closed subspace $J$ of a Banach space is an $\m$-ideal if and only if the following condition holds: if $B_1$, $\ldots$, $B_n$ are open balls with $B_1\cap \cdots\cap B_n\neq \emptyset$ and $B_i\cap J\neq \emptyset$ for all $i$, then $B_1\cap \cdots\cap B_n \cap J\neq \emptyset$. Motivated by this characterization, we propose the following definition for an $\m$-ideal of a ring.

\begin{definition}
\label{mid}
An ideal \(J \) of a ring $R$ is called a ring \emph{\(\m\)-ideal} if for a finite collection of nonzero ideals $I_1, I_2$, $\dots$, $I_n $ ($n\geqslant 2$) of $R$ with  $\bigcap_{k=1}^n I_k\neq 0$, the condition \(J \cap I_k \neq 0\) for each \(k\), where $k\in \{1, \ldots, n\}$, implies that
\[	J \cap \left( \bigcap_{k=1}^n I_k \right) \neq 0.
\]
\end{definition}

The analogy is lattice-theoretic rather than metric. The ideals of $R$, ordered by inclusion, form a complete lattice whose meet is intersection. The preceding definition says that a distinguished ideal $J$ satisfies the following nonvanishing condition for finite intersections: whenever a finite family of test ideals has nonzero intersection and each intersects $J$ nontrivially, their common intersection also intersects $J$ nontrivially. This is the same finite-intersection pattern as in the Alfsen--Effros ball criterion, with nonzero ideals replacing open balls. The facts proved below that essential ideals and minimal ideals both satisfy the condition, together with the reduction to two test ideals, provide algebraic evidence that the analogy reflects intrinsic ideal-lattice structure rather than merely formal terminology.

For brevity, henceforth we shall usually simply use the 
expression $\m$-ideal for `ring $\m$-ideal', and only add the `ring' adjective where there is a possibility of confusion with the Banach space notion (that is, in Banach space settings).  

Unless explicitly stated otherwise, rings are associative and need not have an identity. The preceding definition and all arguments involving only ideal intersections 
apply whether $R$ is unital or not. Whenever an identity is required, unitality is included in the statement of the result. In a few of the results, we require the ring to be commutative.   We confess that we are sometimes a little cavalier 
about the trivial ideal $(0)$.   For example, we generally do not regard it as a {\em proper} ideal, and  we sometimes leave it to the reader to handle separately the $(0)$ ideal case of a statement.  Since this is the ``trivial case'' it should ordinarily  present no difficulty to resolve, or to  determine the intended interpretation of the statement in the trivial case.

\bigskip

Although our choice of taking $n$ ideals $I_1$, $\ldots$, $I_n$ in Definition \ref{mid} was motivated by the characterization of $\m$-ideals of Banach spaces, it is, however, possible to restrict it to $n=2$.  
\begin{proposition}
In any ring, the following are equivalent:
\begin{enumerate}
\item $J$ is an $\m$-ideal.
		
\item  For every  nonzero ideals $I_1,$ $I_2$ of $R$ such that $I_1 \cap I_2\neq 0$, the conditions $J\cap I_1 \neq 0$ and $J\cap I_2\neq 0$ imply that $J\cap (I_1 \cap I_2) \neq 0$.
\end{enumerate}
\end{proposition}

\begin{proof}
Obviously, (1) implies (2). To show (2) implies (1), we use induction on $n$. By (2), the property holds for $n=2$. Assume that the property holds for $n$. Let $I_1,\ldots,I_{n+1}$ be a collection of nonzero ideals of $R$ such that $ \bigcap_{k=1}^{n+1} I_k \neq 0$, and for each $k$ we have $J\cap I_k\neq 0$. By inductive hypothesis, we have $J\cap (\bigcap_{k=1}^{n} I_k ) \neq 0$. Note that \[\left(\bigcap_{k=1}^{n} I_k\right) \cap I_{n+1}= \bigcap_{k=1}^{n+1} I_k \neq 0\]  and $J \cap I_{n+1} \neq 0$. So by assumption, we conclude that \[J \cap \left(\bigcap_{k=1}^{n+1} I_k \right)= J \cap\left(\bigcap_{k=1}^{n} I_k\right) \cap I_{n+1}\neq 0.\qedhere\]
\end{proof}

\begin{remark}  \label{cke} According to \cite[Theorem 5.9]{AE72}, the characterization of $\m$-ideals of Banach spaces can be restricted to three open balls, but not to two. We just saw that the situation is different for rings.

We note that if the ring $R$ is a Banach space, then an \(\m\)-ideal of  $R$
in the sense above need not be an \(\m\)-ideal of  $R$ in the Banach space sense, nor vice versa.  Indeed, suppose that $R = C(K)$, where an $\m$-ideal
$J$ in the latter sense must be the subspace $J_E$ of functions vanishing on a closed subset  $E \subset K$ \cite{HWW93}.  Similarly, take $I_k = J_{E_k}$ for  closed subsets $E_1, E_2$ in $K$.   Note that $J_{\bigcup_k \, E_k} = 
\bigcap_k \, I_k$.  
It is easy to find examples where 
$\bigcup_k \, E_k \neq K$, so that $\bigcap_{k=1}^2 I_k\neq 0$, and $E \cup E_k \neq K$ for each $k$, but $E \cup (\bigcup_k  E_k) = K$. That is, \[J \cap \left(\bigcap_{k=1}^2 I_k\right) = 0.\]
Indeed, this happens in $K = [0,1]$.  Thus $J$ is not a ring $M$-ideal. 
See also Theorem \ref{ck}.

Conversely, consider $A = l^\infty$, with $J$ the finitely supported sequences.   This intersects every nonzero ideal (the noncommutative analogue is contained in every 
nonzero ideal), so it is a ring $\m$-ideal of  $R$
but not a Banach space $\m$-ideal since it is not closed.
\end{remark}

As indicated in the introduction, our notion of 
$\m$-ideals  is a generalization of essential ideals; hence, naturally, these are our first examples.

\begin{proposition}\label{eim}
Every essential ideal of a ring is an \( \m \)-ideal.
\end{proposition}

\begin{proof}
Suppose that $J$ is an essential ideal (and hence nonzero) of a ring $R$. Let $I_1$ and $I_2$ be two nonzero ideals of $R$ such that $I_1\cap I_2\neq 0$ and $J\cap I_k\neq 0$ for each $k$. Since $J$ is essential, we have $J\cap(I_1\cap I_2)\neq 0$.
\end{proof}

\begin{remark}
The reader may wonder whether every 
$\m$-ideal is essential. To see this is not true, consider the ring $\mathds{Z}_{12}$ and the ideal $J:= (3)$. Take $I_1:= (2)$ and $I_2:=(6)$. It is easy to see that $(3)$ is an $\m$-ideal. Note that we cannot include $(4)$ as $I_3$. However, $(3)$ is not an essential ideal because $(3)\cap (4)=(0).$     
\end{remark}

We claim that $\m$-ideals are abundant, and the following proposition will begin to validate this. Of course, we shall see more examples as we proceed.

\begin{proposition}\label{bpmi}
In a ring $R$, the following hold.
\begin{enumerate}
\item The zero ideal is an $\m$-ideal.

\item The maximal ideal of a nontrivial 
 unital local ring is an \( \m \)-ideal.

\item\label{aid} Every  ideal of an integral domain is an \(\m\)-ideal.

\item If $R$ is an algebra and $e$ is a central element 
which is an `algebraically minimal idempotent' in $R$ (so $eR = \bF e$) then $Re$ is an \(\m\)-ideal. 

\item If $R$ is a unital ring and $e$ is a central element 
which satisfies $eR  \setminus (0) \subset R^{-1} e$ (or  $eR  \setminus (0) \subset({\mathcal Z}(R))^{-1} e$) then $Re$ is an \(\m\)-ideal. 

\item\label{minm} If $J$ is a minimal ideal in $R$,  then
$J$ is an $\m$-ideal.
\end{enumerate}
\end{proposition}
	
\begin{proof}
The assertion (1) is vacuously true.
The properties (2) and (3)  follow from Proposition \ref{eim}.
For (4), observe that if $R$ is an algebra and $e$ is a central element 
which is a minimal idempotent in $R$, and $J$ is an ideal in $R$, then $J \cap eR \neq 0$ if and only if $e \in J$.   Now it is clear $eR$ is an \(\m\)-ideal. 
The proof of (5)  is similar to the previous proof.   An ideal $I$ intersects $eR$ if and only if it contains $er$ for an invertible $r$, and so if and only if $e \in I$.  Now it is clear $eR$ is an \(\m\)-ideal. 
The property (6) in fact generalizes the last couple of items.  If $J$ is a 
minimal ideal in $R$  and $K$ is another nonzero ideal, then $J \cap K\neq 0$ implies that $J \cap K = J$.   The result is clear from this. 
\end{proof}

\begin{remark}
Note that a zero ideal of a ring, by definition, is not essential, but it is an $\m$-ideal. In the ring $\mathds{Z}$ of integers, for each prime $p$, the ideal $(p)$ is an essential ideal, and hence, an $\m$-ideal by Proposition \ref{bpmi}(\ref{aid}).   However, $\bigcap_{p}(p)=(0)$ is not essential but is an $\m$-ideal. Although $\m$-ideals are a generalization of essential ideals, interestingly enough, minimal ideals are very far from being essential.

Jacobson radicals are generally not essential. For example, let us consider the ring $\mathds{Z}_{24}$. In this ring, $\mathrm{Jac}(\mathds{Z}_{24})=(6)$, an $\m$-ideal; however, it is not essential because $(6)\cap  (8)=(0)$.    
\end{remark}

\begin{proposition}
\label{new} 
The intersection of a finite number of $\m$-ideals is an $\m$-ideal. \end{proposition}
	
\begin{proof}
Here is a proof for two $\m$-ideals.  Suppose that $J_1$ and $J_2$ be two $\m$-ideals of a ring $R$. Let $I_1$ and $I_2$ be two nonzero ideals of $R$ such that $I_1\cap I_2\neq 0$ and $(J_1\cap J_2)\cap I_k\neq 0$ for each
$k$. This, in particular, implies that $J_2 \cap I_k\neq 0$ for each $k$. Since $J_2$ is an $\m$-ideal, we obtain $J_2\cap (I_1\cap I_2)\neq 0$. Let us now consider the ideals $J_2\cap I_1$ and $J_2\cap I_2$. Notice that 
\[\bigcap_{k=1}^2 (J_2\cap I_k)=J_2\cap (I_1\cap I_2)\neq 0,\]
and $J_1\cap (J_2\cap I_k)\neq 0$ for each $k$. Since $J_1$ is also an $\m$-ideal, we have that
\[(J_1\cap J_2)\cap (I_1\cap I_2)=J_1\cap \left( \bigcap_{k=1}^2 (J_2 \cap I_k) \right)\neq 0.\qedhere\]
\end{proof}

The intersection of an infinite number of $\m$-ideals is not necessarily an $\m$-ideal. Enumerate $\mathbb{Q}\cap[0,\tfrac12]$ as $\{r_1,r_2,\ldots\}$, set $E_n=\{r_1,\ldots,r_n\}$, and let $J_n=J_{E_n}$ in $C([0,1])$ in the notation of Remark  \ref{cke}. Each $J_n$ is essential: if $0\neq f$ belongs to a nonzero ideal, then the open set on which $f$ is nonzero contains a point outside the finite set $E_n$, and multiplying $f$ by a suitable continuous function vanishing on $E_n$ gives a nonzero element of the intersection with $J_n$. By continuity,
\[
\bigcap_{n\geq 1}J_n=J_{[0,1/2]}=:J.
\]
This ideal is not an $\m$-ideal. Indeed, put
\[
F_1=[0,\tfrac14]\cup[\tfrac12,\tfrac34],\qquad
F_2=[0,\tfrac14]\cup[\tfrac34,1],\qquad I_i=J_{F_i}.
\]
Then $I_1\cap I_2\neq 0$, $J\cap I_i\neq 0$ for $i=1,2$, but $J\cap I_1\cap I_2=0$.

\begin{proposition}\label{jk}
Let $R$ be a unital ring, let $J$ be an $\m$-ideal of $R$, and let $K\subseteq J$ be a direct-summand ideal of $R$. Then $J/K$ is an $\m$-ideal of $R/K$. 
\end{proposition}
	
\begin{proof}
Write $R=K\oplus L$ as a direct sum of ideals. Since $K\subseteq J$, we have $J=K\oplus(J\cap L)$, and the quotient map identifies $R/K$ with $L$ and $J/K$ with $J\cap L$. Let $I_1$ and $I_2$ be nonzero ideals of $L$ such that $I_1\cap I_2\neq 0$ and $(J\cap L)\cap I_i\neq 0$ for $i=1,2$. The ideals $I_1$ and $I_2$ are also ideals of $R$. Since $J$ is an $\m$-ideal of $R$,
\[
0\neq J\cap I_1\cap I_2=(J\cap L)\cap I_1\cap I_2.
\]
Thus $J\cap L$, and hence $J/K$, is an $\m$-ideal.
\end{proof}

The direct-summand hypothesis in the last result cannot be omitted. For $R=\mathds{Z}$, $J=2\mathds{Z}$, and $K=30\mathds{Z}$, the ideal $J$ is essential and hence an $\m$-ideal, whereas $J/K$ corresponds to $(2)$ in $\mathds{Z}_{30}$ and is not an $\m$-ideal: taking $I_1=(3)$ and $I_2=(5)$ gives $(2)\cap I_1=(6)\neq 0$, $(2)\cap I_2=(10)\neq 0$, and $I_1\cap I_2=(15)\neq 0$, but $(2)\cap I_1\cap I_2=0$.

For nonunital rings, several other items in the `calculus of Banach space $\m$-ideals' carry over without adjoining an identity.   For example, 
if $J$ is a ring $\m$-ideal of $R$, and $K$ is a ring $\m$-ideal of $J$, then
$K$ is a ring $\m$-ideal of $R$ (see Proposition \ref{krs}).

Many results about $\m$-ideals in general rings $R$ can be deduced from the unital case by unitization and from the next corollary.  This trick usually works so long
as there exists  a unitization $R^1$ in which $R$ is essential. We call this an essential unitization.  However such existence follows e.g.\ from \cite[Lemma 1]{OJ81}. They choose an ideal $K$ which is maximal subject to $K\cap I=0$, and prove that $I+K$ is essential. Applying exactly that argument to $I=R$ inside the Dorroh unitization $D(R)$, then passing to $D(R)/K$, gives an essential unitization.  Indeed  if $L/K$ is an ideal of $D(R)/K$ disjoint from  $(R+K)/K$, then $L\cap R=0$.  Maximality then forces $L=K$.

Note that a nonunital ring $R$ is an $\m$-ideal of this $R^{1}$, since it is
an essential ideal. More generally, we have from the observations above
that:

\begin{corollary}
\label{uni} 
If $R$ is a nonunital ring then the $\m$-ideals of $R$ are the  $\m$-ideals of  the unitization $R^1$ above, 
that are contained in $R$.  They are also the intersections of $R$ with 
the  $\m$-ideals of  $R^1$.
\end{corollary}

\begin{proof}  The first assertion follows from 
the observations above the corollary.  If $J$ is
an $\m$-ideal of  $R^1$ then $J \cap R$ is an $\m$-ideal of $R^1$ by e.g.\
Proposition \ref{new}. 
Hence, it is an $\m$-ideal of $R$ by the first 
assertion. 
\end{proof}

Thus the $\m$-ideals of a nonunital ring are compatible with passage to its unitization in the precise sense stated in the corollary. 

We note that if $J_k$ is an $\m$-ideal in $R_k$ for $k = 1, \ldots, m$, then $\oplus_k \, J_k$ need not be an $\m$-ideal in $\oplus_k \, R_k$.  Indeed $J_1 \oplus R_2$ need not be an $\m$-ideal in $R_1 \oplus R_2$.  See Theorem \ref{fix}.  Thus, the sum of two $\m$-ideals need not be an $\m$-ideal.  
Similarly, $J_1 \otimes J_2$ need not be an $\m$-ideal in $R_1 \otimes R_2$, even if $J_2 = R_2$.

\begin{remark}
Based on the above abundance of examples, one may be tempted to believe that all nonzero ideals of rings are $\m$-ideals. However, this is easily seen not to be true. Consider the simple example of the ring $\mathds{Z}_{30}$ and the ideal $J:= (2)$. If we take $I_1:= (3)$ and $I_2:=  (5)$, then $J\cap I_1= (6)$, $J\cap I_2= (10)$, and $I_1\cap I_2= (15)$. However, $J\cap (I_1\cap I_2)=(0)$.    
\end{remark}

\section{Characterizations and Applications}

Our next goal is to give a characterization of $\m$-ideals of rings, and this will illuminate further why minimal ideals are $\m$-ideals. This result and its consequences are an ideal-theoretic adaptation from \cite{GCM25}. Before we state the result, we need a definition.

\begin{definition}
We say that an ideal $I$ of a ring $R$ is \emph{relatively irreducible}\footnote{\textit{cf}. An ideal $I$ of $R$ is called \emph{irreducible} if for any two ideals $J$, $K$ of $R$ and $I=J\cap K$ implies that either $J=I$ or $K=I$.  Our `relatively irreducible' is a form of `prime-ness'} if for  all  
ideals $J$ and $K$ of $R$ with $J\subseteq I$, $K\subseteq I$, and $J\cap K=0$,   either $J=0$ or $K=0$. 
\end{definition}

\begin{theorem}\label{chmi}
An ideal $I$ of a ring $R$ (not necessarily unital) is an $\m$-ideal if and only if either it is essential or relatively irreducible.
\end{theorem}

\begin{proof}
Clearly, if $I$ is essential, then $I$ is an $\m$-ideal by Proposition \ref{eim}. Now suppose that $I$ is relatively irreducible. Let $J$ and $K$ be two nonzero ideals of $R$ such that $J\cap K\neq 0$, $I\cap J\neq 0$, and $I\cap K\neq 0$. Since $I$ is relatively irreducible,
this implies
\[I\cap (J\cap K)=(I\cap J)\cap (I\cap K)\neq 0.\]
Hence $I$ is an $\m$-ideal. Conversely, suppose that $I$ is an $\m$-ideal and $I$ is not essential. We show that $I$ is relatively irreducible. Since $I$ is not essential, there exists a nonzero ideal $L$ in $R$ such that $I\cap L=0$. Let $J$ and $K$ be nonzero ideals of $R$ such that $J\subseteq I$ and $K\subseteq I$. We claim that $J\cap K\neq 0$. Let $J':=J+L$ and $K':= K+L$. Then $J'\cap K'\supseteq L \neq 0$. Applying modular law, we obtain 
\[I\cap J'=I\cap(J+L)=J+(I\cap L)=J\neq 0.\]
Similarly, we get $I\cap K'=K\neq 0$. Since $I$ is a $\m$-ideal, we have
\[ J\cap K=(I\cap J')\cap (I\cap K')=I\cap (J'\cap K')\neq 0.\qedhere\]
\end{proof}

\begin{remark}
Every minimal ideal is relatively irreducible. Indeed, let $I$ be a minimal ideal of $R$. Suppose that $J\subseteq I$, $K\subseteq I$, and $J\cap K=0$. Since $I$ is minimal, the only case we need to care about is when $J=I$ and $K=I$. But then $J\cap K=I\neq 0$, a contradiction. Thus, the assertion in Proposition \ref{bpmi}(\ref{minm}) is no longer surprising. However, note that there exist $\m$-ideals that are neither essential nor minimal (for example, consider the ideal $3\mathds{Z}_{12}$ in the ring $\mathds{Z}_{12}$). So, Theorem \ref{chmi} can not be improved further.   
\end{remark}

For the rest of this section, we shall see some applications of the above characterization theorem. 
Recall that the sum of minimal ideals of a ring $R$ is called the \emph{socle} of $R$, and we denote it by $\mathrm{Soc}(R)$. It is well-known that the socle of a ring $R$ is the intersection of all essential ideals of $R$. Now, Theorem \ref{chmi}, tells us when $\mathrm{Soc}(R)$ is an $\m$-ideal. 

\begin{theorem}
Suppose $n$ denotes the number of minimal ideals of a ring $R$. If $n\leqslant 1$, the $\mathrm{Soc}(R)$ is an $\m$-ideal. If $n\geqslant 2$, the $\mathrm{Soc}(R)$ is an $\m$-ideal if and only if it is essential.
\end{theorem}

\begin{proof}
If 
$R$ has no minimal ideals, then $\mathrm{Soc}(R)$ is trivially an $\m$-ideal. If $R$ has only one minimal ideal $I$, then $\mathrm{Soc}(R)=I$, is relatively irreducible and hence, by Theorem \ref{chmi}, is an $\m$-ideal. If $R$ has more than one minimal ideal, then $\mathrm{Soc}(R)$ is not relatively irreducible, as the intersection of any two distinct minimal ideals $I$ and $J$ is $0$. Indeed: $I\cap J\subseteq I$, whence $I\cap J=0$ or $I\cap J=I$. But if $I\cap J=I$, then $I\subseteq J$, whence $I=J$. Therefore, if $R$ has more than one minimal ideal, then by Theorem \ref{chmi}, $\mathrm{Soc}(R)$ is an $\m$-ideal precisely when it is essential.
\end{proof} 

Using Theorem \ref{chmi}, we shall now characterize $\m$-ideals of rings $\mathds{Z}_n$. 

\begin{theorem}\label{charzn}
Consider a ring $\mathds{Z}_n$ with  $n=p_1^{m_1}\cdot \ldots \cdot p_k^{m_k}$
(the prime decomposition). Then a nontrivial ideal
$$I:=\left(p_1^{m'_1}\cdot \ldots \cdot p_k^{m'_k}\right)$$
is an $\m$-ideal  if and only if there are not $i$, $j$, $s \in \{1,\ldots, k\}$ such that $i\neq j$, $m'_i < m_i$, $m'_j < m_j$, and $m'_s=m_s$. 
In particular, $I$ is essential if and only if for every $i$ we have $m'_i < m_i$.
\end{theorem}

\begin{proof}  
We first prove that $I$ is essential if and only if for every $i$ we have $m'_i < m_i$.
If $m'_i < m_i$, given an ideal $J\neq 0$ with $J=(p_1^{m''_1}\cdot \ldots \cdot p_k^{m''_k})$, we obtain 
$$I \cap J= \left(p_1^{\operatorname{max}\{m'_1,m''_1\}}\cdot \ldots \cdot p_k^{\operatorname{max}\{m'_k,m''_k\}}\right).$$
Since $J\neq 0$, there exists $j\in \{1,\ldots, k\}$ such that $m''_j < m_j$ and so $I\cap J\neq 0$ since $p_j$ appears with exponent $\operatorname{max}\{m'_j,m''_j\}< m_j$ in the unique factorization of the generator of $I \cap J$. Conversely, if $I$ is essential, then for every $i\in\{1,\ldots, k\}$ we must have $$I \cap \left(p_1^{m_1}\cdot \ldots \cdot p_i^{m_i-1} \cdot \ldots \cdot p_k^{m_k}\right) \neq 0,$$ and thus $\operatorname{max}\{m_i-1, m'_i\} < m_i$. So, it must be $m'_i < m_i$. 

Let us now consider the case in which $I$ is not essential, and thus there exists $s\in \{1,\ldots,k\}$ such that $m'_s=m_s$. Suppose there are not $i$, $j\in \{1,\ldots, k\}$ such that $i\neq j$, $m'_i < m_i$, $m'_j < m_j$. 
If $m'_j=m_j$ for every $j\neq i$ and $m'_i=m_i -1$, then $I$ is a minimal ideal and thus, an $\m$-ideal. Indeed, if an ideal \[J=\left(p_1^{m''_1}\cdot \ldots \cdot p_k^{m''_k}\right)\] is  contained in $I$ we must have $m_j \leqslant m''_j$ (so $m_j = m''_j$) for every $j\neq i$ and $m_i-1 \leqslant m''_i$ and thus either $I=J$, or $J=0$. Otherwise, there must exist $i\in \{1, \ldots, k\}$ such that $m'_j=m_j$ for every $j\neq i$ and $m'_i < m_i-1$. Let $J$, $L$ be ideals of $\mathds{Z}_n$ such that $J\cap L \neq 0$, $I \cap J \neq 0$ and $I \cap L \neq 0$. Since $I \cap J \neq 0$, the prime $p_i$ must appear in the factorization of the generator of $J$ with an exponent strictly smaller than $m_i$. And analogously for the generator of $L$. This implies that $I \cap J\cap L \neq 0$ since the exponent of the prime $p_i$ in the unique factorization of the generator of $I \cap J\cap L$ is strictly less than $m_i$. This proves that $I$ is an $\m$-ideal. 

Suppose now that there exists $i,j\in \{1,\ldots, k\}$ such that $i\neq j$, $m'_i < m_i$, $m'_j < m_j$ and consider the following ideals: 
$$J=\left(p_1^{m_1}\cdot \ldots \cdot p_i^{m_i} \cdot \ldots \cdot p_j^{m'_j} \cdot \ldots \cdot p^0_s\cdot \ldots \cdot p_k^{m_k}\right) \quad \text{ and } \quad L=\left(p_1^{m_1}\cdot \ldots \cdot p_i^{m'_i} \cdot \ldots \cdot p_j^{m_j} \cdot \ldots \cdot p^0_s\cdot \ldots \cdot p_k^{m_k}\right).$$
Then, since $m_s>0$, we have 
$$J\cap L=\left(p_1^{m_1}\cdot \ldots \cdot p_i^{m_i} \cdot \ldots \cdot p_j^{m_j} \cdot \ldots \cdot p^0_s\cdot \ldots \cdot p_k^{m_k}\right) \neq 0.$$
Moreover, since $m'_j < m_j$, we have $$I \cap J=\left(p_1^{m_1}\cdot \ldots \cdot p_i^{m_i} \cdot \ldots \cdot p_j^{m'_j} \cdot \ldots \cdot p_s^{m_s}\cdot \ldots \cdot p_k^{m_k}\right)\neq 0$$  and analogously, $I \cap L\neq 0$ because $m'_i < m_i$. But clearly $I\cap J\cap L=0$. So, $I$ is not an $\m$-ideal.
\end{proof}

\begin{remark}
Let us now apply the above result to show that the product of $\m$-ideals may not be an $\m$-ideal.	Consider the ring $\mathds{Z}_{900}$. Since $900=2^2\cdot 3^2 \cdot 5^2$, by Theorem \ref{charzn}, the ideals $(6)=(2\cdot 3)$ and $(10)=(2\cdot 5)$ are $\m$-ideals. However, again by Theorem \ref{charzn}, $(6)\cdot(10)=(60)=(2^2\cdot 3 \cdot 5)$ is not a $\m$-ideal. 

Also, notice that a maximal ideal of a ring may not be an $\m$-ideal.	Consider the ring $\mathds{Z}_{180}$. The ideal $(5)$ is maximal but is not an $\m$-ideal by Theorem \ref{charzn}. Also, the Jacobson radical, $\mathrm{Jac}(R)$, may not be an $\m$-ideal. For example, $\mathrm{Jac}(\mathds{Z}_{180})=(2\cdot 3 \cdot 5)=(30)$, is not a $\m$-ideal, again thanks to Theorem \ref{charzn}.  
\end{remark}

Let $C$ be an $R$-module and $A$ be a submodule of $C$. Recall from \cite{GW04} that $A$ is said to be
\emph{essentially closed in} $C$ provided $A$ has no proper essential extensions within
$C$. Motivated by this, we now introduce a similar notion in our context and show its relation with $\m$-ideals.

\begin{definition}
We say an ideal $I$ is called \emph{essentially closed} (resp.\ \emph{$\m$-closed}) in a ring $R$, if $I$ is an essential ideal  (resp.\ an $\m$-ideal) in an ideal $J$ of $R$, then $I=J$.     
\end{definition}

\begin{theorem}\label{charmuclosed}
Let $R$ be a ring and let $I$ be a proper ideal of  $R$. The following are equivalent:
\begin{enumerate}
\item $I$ is $\m$-closed;

\item $I$ is essentially closed and not relatively irreducible.
\end{enumerate}
\end{theorem}

\begin{proof}
(1)$\Rightarrow$(2): Let $I$ be $\m$-closed. Since every essential ideal is an $\m$-ideal, clearly $I$ is essentially closed. Moreover, $I$ is not relatively irreducible. Indeed, if $I$ is  relatively  irreducible in $R$, by Theorem \ref{chmi}, $I$ is an $\m$-ideal of $R$. Since $I\neq R$, this implies that $I$ is not $\m$-closed, which is a contradiction. Hence, $I$ is not relatively irreducible in $R$.
		
(2)$\Rightarrow$(1): Assume that $I\neq R$ is essentially closed and not relatively irreducible. Let $J$ be an ideal of $R$ such that $I$ is an $\m$-ideal of $J$. We show that $J=I$. By Theorem \ref{chmi},  $I$ is either relatively irreducible in $J$ or it is essential in $J$. If $I$ is relatively irreducible in $J$, then $I$ is relatively irreducible in $R$ as well, which is a contradiction. If $I$ is essential in $J$, then $I=J$ because $I$ is essentially closed. Hence, $I$ is $\m$-closed. 
\end{proof}

We shall now see an example of a ring with both essential and minimal $\m$-ideals.

\begin{theorem}\label{ck} 
Let $K$ be a compact Hausdorff space. 
The nonzero ring $\m$-ideals in $C(K)$ are precisely either the essential ideals or the minimal ideals singly generated by the characteristic function
of an isolated point. 
\end{theorem}
	
\begin{proof} As noted in Remark \ref{cke}, the closed ideals, which are the Banach space $\m$-ideals of the ring $C(K)$, are the subspaces $J_E$ of functions vanishing on a closed subset  $E \subset K$. 
If $E = K \setminus \{ x \}$
for $x \in K$ then $J_E$ satisfies the conditions of Proposition \ref{bpmi}(\ref{minm}), so is an $\m$-ideal.  
Let $E \subset K$ be a closed subset such that 
 $J_E$ is a ring $\m$-ideal.  
If $E =  K \setminus U$ and $U$ is not a singleton (the case 
above), then we claim that $U$ is dense.
For if not, then there is a nonempty open set 
$W \subset E$.   Any points $x \neq y$ in $U$ have disjoint neighborhoods $W_1, W_2$ inside  $U$.  Set $V_k = W \cup W_k$, and $E_k = K \setminus V_k$, with $I_k = J_{E_k}$.  Then $\bigcup_k \, E_k \neq K$, so that 
$\bigcap_{k=1}^2 I_k\neq 0$, and $E \cup E_k \neq K$ for each $k$, but $E \cup \bigcup_k \, E_k = K$. That is, $I \cap \bigcap_{k=1}^2 I_k = 0$.
Thus, $U$ is dense. 

It is known \cite{GJ60}  that a nonzero ideal $I$  in $C(K)$ is essential if and only if $\bar{I}$ is essential and if and only if the interior of $Z(I) = Z(\bar{I})$ is empty. 
Here \[Z(I) = \bigcap_{f \in I} \, f^{-1}(\{ 0 \}).\]   If $\bar{I} = J_E$ then $E = Z(\bar{I})$, and the interior of $E$ is empty if and only if 
$U = K \setminus E$ is dense.   See also II.5.4.7 in \cite{Bla06}. If $I$ is a ring $\m$-ideal in $C(K)$
which is not of the form in the first paragraph of this proof, then 
$\bar{I}$ is a ring $\m$-ideal too.  Indeed if $I$ is relatively irreducible then so is 
$\bar{I}$ by e.g.\ Lemma \ref{l2i} below.
By the above $\bar{I}$ is an essential ideal, so that $I$ is an essential ideal.

Thus, the nonzero ring $\m$-ideals in $C(K)$ are the essential ideals or the minimal ideals singly generated by the characteristic function 
of an isolated point.  
\end{proof}

Thus, for example, the ring $\m$-ideals in $C([0,1])$ are exactly the essential ideals.
This is because $[0,1]$ has no isolated points.

\section{In C*-algebras}

The corresponding problem for a noncommutative 
 C*-algebra $B$ is also interesting. 
 Note that a ring ideal  in a unital algebra is clearly an algebra ideal.   Similarly, the closure of a ring ideal  in a Banach algebra with bounded approximate identity (bai) is an algebra ideal. 

As a first step, one can show that an ideal J in $B$ is essential in the algebraic sense if and only if $\bar{J}$ is essential in the C*-sense.  Indeed, this follows very easily from:

\begin{lemma} \label{l2i} For nonzero two-sided  ring ideals $I, J$ in a  C*-algebra $B$ we have $\overline{I \cap J} = \bar{I} \cap \bar{J}$.
\end{lemma}
	
\begin{proof} Any  bai for $\bar{I}$ may be approximated to give a  bai $(e_t)$ for $I$ which is also a
 bai for $\bar{I}$.  If $x \in \bar{I} \cap \bar{J}$, with bounded 
approximation $j_s \to x$ with 
$j_s \in J$, then $$x = \lim_t e_t x = \lim_t \lim_s \, e_t j_s.$$ This
is in $\overline{I \cap J}$, since
$$e_t j_s \in I J \subseteq I \cap J \subseteq \overline{I \cap J}.$$
It follows that $\overline{I \cap J} = \bar{I} \cap \bar{J}$.
\end{proof} 

The same proof works in a Banach algebra, provided that 
$\bar{I}$ has a bounded approximate identity.  This must be well-known.        

We recall that a C*-algebra $A$ is {\em prime} if it does not contain two proper closed orthogonal two-sided ideals.
For example, $B(l^2)$ is prime since it has a unique 
proper closed two-sided ideal.   If $A$ contains two nonzero central positive elements with product 0, then $A$  is not prime.  However, this condition does not characterize (non)primeness.  Rather, $A$ is prime if and only if for all nonzero $x$, $y \in A$ we have $xAy \neq 0$ \cite[Proposition II.5.4.5]{Bla06}. It follows that  $A$ is prime as a C*-algebra if and only if it is prime as a ring. Also, $A$ is prime if and only if every nonzero ideal of $A$ is an essential ideal.  In the separable case, primeness is equivalent to the existence of a faithful irreducible representation \cite{Bla06}.  

\begin{proposition}
\label{proC} For a nonzero  ideal $J$ in a   C*-algebra the following conditions are equivalent:
\begin{enumerate}
\item $J$ is relatively irreducible.
		
\item  $\bar{J}$ is a prime C*-algebra
\item $x J y \neq 0$ for nonzero $x, y \in J$.  Or equivalently, $J$ is a prime ring. 
\end{enumerate}
\end{proposition}

\begin{proof} For a closed ideal in a C*-algebra, `relatively irreducible' is precisely the same as being a prime C*-algebra (see \cite[Definition II.5.4.4]{Bla06}).
It follows that for a non-closed ideal $J$ with $\bar{J}$ a prime C*-algebra, Lemma \ref{l2i}  implies that $J$ is relatively irreducible.  Indeed for nonzero ideals $I, K$ in $J$ with $I\cap K=0$ we have $\bar{I} \cap \bar{K} = 0$, giving a 
contradiction.  So (2) implies (1).

The equivalence of (2) and (3) uses  \cite[Proposition II.5.4.5]{Bla06} and Lemma \ref{l2i}.  If $xJy =  0$, then by continuity $x \bar{J} y = 0$, while if $x \bar{J} y = 0$ then of course $x J y =  0$.  Hence, if $\bar{J}$ is a prime C*-algebra, then $xJy \neq 0$ for nonzero $x, y \in J$.  Conversely, suppose that (3)
holds.   Suppose that 
 $I, K$ are nonzero closed ideals in $\bar{J}$.  Then 
$I \cap J \neq 0$ or else Lemma \ref{l2i} gives a contradiction.  Similarly, $K \cap J \neq 0$.  
Suppose that $0 \neq x \in I \cap J,
0 \neq y \in K \cap J$.  Then $x \bar{J} y \subset I \cap K$.  If $I \cap K = 0$ we obtain the contradiction that $x$ or $y$ is 0.   Thus $I \cap K \neq 0$.  Thus $\bar{J}$ is a prime C*-algebra. 

Finally, suppose that (1) holds.  The argument above shows that if 
$I, K$ are nonzero closed ideals in $\bar{J}$ then 
$I \cap J \neq 0$ and  $K \cap J \neq 0$.  So $I \cap K \cap J 
\neq 0$. By Lemma  \ref{l2i} we have 
$I \cap K \cap \bar{J} \neq  0$. So 
  $\bar{J}$ is a prime C*-algebra, and we have (2).
\end{proof}

The following result now follows from Theorem \ref{chmi}.

\begin{corollary}
\label{nck}
The nonzero ring $\m$-ideals in a C*-algebra $B$ are precisely either the essential ideals or the ideals $J$ whose closure is a prime  C*-algebra.   Such ideals $J$ are exactly the ideals which are a prime ring. 
\end{corollary}

\begin{corollary} \label{ncs} If $I$ is a ring $\m$-ideal in a unital C*-algebra $A$ and if $I \subseteq B \subseteq A$ with $B$ a
subalgebra, then $I$ is a ring $\m$-ideal in $B$.
\end{corollary}

\begin{proof}  If $\bar{I}$ is a prime C*-algebra, then $I$ is 
relatively irreducible  in $A$ by the proposition.
 So it is relatively irreducible  in $B$, and hence is a ring $\m$-ideal in $B$.    Indeed if $J,K$ are
nonzero ideals of  $B$ contained in $I$, then by Proposition \ref{proC}(3) we have $0 \neq JIK \subset  J\cap K$.
On the other hand, if $I$ is essential in $A$, then so is $\bar{I}$.
The left regular representation on $\bar{I}$ is faithful on $A$
(this is a well-known characterization of essential ideals in C*-algebras, and is a nice exercise in that subject). Hence, it is also faithful on the closure $\bar{B}$ of $B$.  However, this easily implies that  $\bar{I}$ is essential in $\bar{B}$. 

The argument for Lemma  \ref{l2i} now works to show that for a nonzero ideal $J$ in $B$ we have 
$\overline{I \cap J} = \bar{I} \cap \bar{J} \neq 0$ in $\bar{B}$, and so $I \cap J  \neq 0$.  That is, $I$ is essential in $B$. 
\end{proof}

Of course, the results above apply to all finite-dimensional semisimple algebras by Wedderburn's theorem.

The last result has a variant for von Neumann regular rings.

 \begin{corollary} \label{ncs2} If $I$ is an $\m$-ideal in a 
 regular  ring $A$ and if $I \subseteq B \subseteq A$ with $B$ a subring, then $I$ is an $\m$-ideal in $B$.
\end{corollary}

\begin{proof}   Suppose that  $I$ is 
relatively irreducible  in $A$ (in a regular ring this is equivalent     to being a prime ring).  Then  $I$ is 
relatively irreducible  in $B$, and hence is a ring $\m$-ideal in $B$.
On the other hand, if $I$ is essential in $A$, then the left regular representation $L$ on $I$ is faithful on $A$. Indeed, if $a I = 0$ then $AaAI = 0$.  However, $AaAI = I \cap (AaA)$ since we are in a regular  ring, and this is nonzero if $a \neq 0$
since $I$ is essential. So $a = 0$. 
Hence, $L$ is also faithful on $B$.  However, this easily implies that  $I$ is essential, hence an $\m$-ideal, in $B$. 
\end{proof} 

We conclude this section with a couple of examples of $\m$-ideals in some famous nonselfadjoint operator algebras and function algebras.

\begin{proposition}
Every ideal in the upper triangular matrices $\cT_n$,
or in the disk algebra or in $H^\infty(\mathds{D})$, is  an 
$\m$-ideal. \end{proposition}
	
\begin{proof} 
Every ideal in the upper triangular matrices $\cT_n$ is essential and hence is an 
$\m$-ideal.   This is because every nonzero ideal $J$ in $\cT_n$ contains $e_{1n}$ (this is easily seen by looking at $e_{1i} J e_{jn}$).

Similarly,  every nonzero ideal in the disk algebra or in $H^\infty(\mathds{D})$ is essential and hence is an 
$\m$-ideal.   Indeed, this follows from Proposition \ref{bpmi} since these are integral domains.  If $I$ and $J$ are nonzero ideals here, then their intersection is nonzero.
This follows since for nonzero $f \in I$ and $g \in J$ we must have $fg$ not identically $0$ on the disk. 
\end{proof}

\section{Some Structural Aspects}

We shall now focus on studying the influence of $\m$-ideals on the structure of rings.

\begin{lemma}
\label{pros}
Suppose that a unital ring $R$ has a nontrivial ideal decomposition $R = I \oplus J$.  Then $I$ is an $\m$-ideal in $R$ if and only if $I$ is relatively irreducible.
\end{lemma}

\begin{proof} Clearly $I$ is not essential, so it is 
an $\m$-ideal in $R$ if and only if it is relatively irreducible.  
\end{proof}

\begin{theorem} \label{fix} Let $n \geq 2$, and $R_1,  \cdots R_n$ be a finite collection of unital rings, with ideals $I_1, \cdots , I_n$ respectively, with
at least one of these is nonzero.
Let  $R =  \prod_{k=1}^n \, R_k$, the direct product, and  $I = \prod_{k=1}^n \, I_k$.
In places below we will  identify $I = \oplus_k \, I_k$, and think of this as a sum of ideals of $R$. 
\begin{enumerate}
\item If all the $I_k$ are nonzero then $I$ is  an $\m$-ideal in $R$ if and only if $I_k$ is essential in $R_k$ for all $k$.
\item Suppose that  $I_k = 0$ for some $k$.  Then   $I$ is an $\m$-ideal in $R$ if and only if $I = I_j$ for some $j$
and  $I_j$ is relatively irreducible.

\item Any  nonzero
$\m$-ideal of $R$ is such a direct product of $\m$-ideals, that is it is of the type in {\rm (1)} or {\rm 
 (2)} above. 
\end{enumerate} \end{theorem}
	
\begin{proof}  This relies on the well-known fact that the ideals of a direct product of a finite number of rings are exactly the direct products of
ideals.   
 
If $I$ is an $\m$-ideal and at least two of the $I_k$ are nonzero then $I$ is essential (from which it follows that all of the $I_k$ are nonzero and essential in $R_k$). 
For otherwise, if $J \cap I= 0$ for  ideal $J \neq 0$,
then $J + I_1$ and $J + I_2$ intersect each other nontrivially, and each intersects $I$ if $I_1$ and $I_2$ are
nonzero, but
their intersection does not  intersect $I$ nontrivially.  
 
(1)\ Suppose that $I_k$ is essential in $R_k$ for all $k$ and that
$J_k$ and $K_k$ are ideals  in $R_k$ for all $k$, with $J_i \cap K_i \neq 0$ say.  Then
$I_i \cap (J_i \cap K_i) \neq 0$.  So $I \cap (\prod_{i=1}^n \, J_i) \cap (\prod_{i=1}^n \, K_i) \neq 0$. Hence
$I$ is an $\m$-ideal.
 
Conversely, if one of the $I_k$ is not essential in $R_k$
then $I$ is not an $\m$-ideal.  For example, suppose that $I_1$ is not essential
with  $J \cap I_1 = 0$ for an ideal $J$ in $R_1$. Then $J \cap I = J \cap I_1 = 0$,
which contradicts the fact in the second paragraph of the proof if $I$ is an $\m$-ideal.

(2)\ If at least two of the $I_k$ are nonzero  then  they must be all nonzero to get an  $\m$-ideal
by the second paragraph.   If all the $I_k$ are zero except one, then 
the proof of  Lemma \ref{pros} gives (2).

(3)\ Suppose that $I$  is a nonzero $\m$-ideal of $\prod_{i=1}^n \, 
R_i$.  Then by the fact in the first paragraph of the proof,
 $I =  \prod_{k=1}^n \, I_k$.
Suppose that $J_1$ and $K_1$ are ideals in $R_1$ with $J_1 \cap K_1, J_1 \cap I_1, I_1 \cap K_1$ nonzero.  Then the intersection of any two of
\[J_1\times 0\cdots\times 0,\; \; \; K_1 \times 0\cdots\times 0, \; \; \; I = I_1 \times \left(\prod_{i=2}^n \, I_k\right)\]
is nonzero.
Since $I$ is an $\m$-ideal, the intersection of these three is nonzero.  Thus $I_1 \cap J_1 \cap K_1 \neq 0$.
Hence $I_1$ is an $\m$-ideal, and similarly so is $I_2, \cdots, I_n$. Thus $I$ is a direct product of $\m$-ideals. 
\end{proof}

\begin{remark}  \label{sdp} Following on from the last result, one may ask if an $\m$-ideal of a subdirect product $R$ of a finite number of rings $R_i$ is a subdirect product of $\m$-ideals of $R_i$?   However, this is false,  as one can see with 
 the $\m$-ideal $(2)$ in $\mathds{Z}_n$ for suitable $n$.

Similarly, it is also easy    to find examples in e.g.\ $\mathds{Z}_{30}$ showing that it is false that in the setting of the last paragraph of a subdirect product $R$ of rings $R_i$,   $\m$-ideals $I_k$ of $R_k$ correspond to an $\m$-ideal of $R$, in analogy to the theorem. \end{remark}

Our next theorem characterizes  direct sum
decomposition into  simple rings in terms of the absence of $\m$-ideals, and it extends Theorem 3 from \cite{OJ81}. To prove the result, we first need a lemma.

\begin{lemma}
\label{fl1}
If \( I \) is a nonzero proper ideal of a unital ring \( R \), then either \( I \) is a direct summand of \( R \) (so is unital and has a complementary ideal in $R$), or there exists a proper \( \m \)-ideal of \( R \) containing \( I \).
\end{lemma}

\begin{proof}
Suppose that \( I \) is 
an ideal of \( R \) which is not a direct summand.
We claim that there exists an \( \m \)-ideal \( K \) of \( R \) such that \( I \subseteq K \).
Let \( \mathcal{C} \) be the collection of all ideals \( K \) of \( R \), which are not a direct summand of $R$, such that \( I \subseteq K \) and \( K \neq R \).  Since $I\in \mathcal{C}$, the set $\mathcal{C}$ is nonempty.
Partially order $\mathcal{C}$ by inclusion. Let $\{J_i\}_{i\in\Lambda}$ be a chain in $\mathcal{C}$ and put $J=\bigcup_{i\in\Lambda}J_i$. Then $J$ is an ideal of $R$ containing $I$. Moreover, $J\neq R$: otherwise $1_R\in J_i$ for some $i$, and hence $J_i=R$, contradicting $J_i\in\mathcal{C}$. The ideal $J$ is not a direct summand of $R$. Indeed, if it were, then $J$ would have an identity $e$; since $e\in J_i$ for some $i$ and $J_i$ is an ideal of $J$, it would follow that $J_i=J$, contradicting the fact that $J_i$ is not a direct summand. Thus $J\in\mathcal{C}$ and is an upper bound for the chain. Zorn's lemma therefore yields a maximal element $K$ of $\mathcal{C}$.

To prove that \( K \) is an \( \m \)-ideal, let \( I_1\) and \(I_2\) be nonzero ideals of \( R \) such that \(  I_1\cap I_2 \neq 0 \) and \( K \cap I_k \neq 0 \) for each \( k \).
If \( K \cap \big(\bigcap_{k=1}^2 I_k\big) = 0 \), then the ideal \( K + \big(\bigcap_{k=1}^2 I_k\big) \) is strictly larger than \( K \) containing \( I \), and \[ K + \big(\bigcap_{k=1}^2 I_k\big) \neq R,\] since $K$ is not a direct summand of $R$. A similar argument shows that \( K + \big(\bigcap_{k=1}^2 I_k\big) \) is not a direct summand.  This contradicts the maximality of \( K \) in \( \mathcal{C} \).
Hence, \( K \cap \big(\bigcap_{k=1}^2 I_k\big) \neq 0 \), proving that \( K \) is an \( \m \)-ideal.
\end{proof}

{\bf Remark.}  The proof shows that proper \( \m \)-ideals have maximal proper \( \m \)-ideal extensions.
 Indeed even in the nonunital case part of the Zorn's lemma argument works, because an upwardly directed union of  \( \m \)-ideals can be shown to be an \( \m \)-ideal.

\begin{theorem}\label{dss}
For any unital ring \( R \), the following are equivalent:
\begin{enumerate}
\item Every proper nonzero  \( \m \)-ideal of \( R \) is a direct summand of \( R \).

\item Every proper nonzero  \( \m \)-ideal of \( R \) is a simple ideal of \( R \) which is also a direct summand. 

\item Every proper ideal of \( R \) is a direct summand of \( R \).

\item \( R \) is a  finite direct sum of simple rings.
\end{enumerate}
\end{theorem}

\begin{proof}
If every nonzero  \( \m \)-ideal of \( R \) is a direct summand of \( R \), then $R$ has no essential ideals.  By Proposition \ref{eim} and Lemma 1 from \cite{OJ81}, every ideal is a direct
summand of $R$.  Hence, (1) implies (3).  Clearly (2) implies (1). For a proof of (3)$\Leftrightarrow$(4), we refer to \cite{Arm68}. 

If (4) holds and $J$ is a proper ideal in $R$, then it is easy to see that $J$ is the direct sum $\oplus_{i \in I} J_i$ of certain of the simple subrings.   Let $J_0$ be one of the simple subrings different from the $J_i$. Suppose that $I$ is not singleton, with $i,j \in I$, $i \neq j$.   Setting $I_1 = J_0 \oplus  J_i, I_2 = J_0  \oplus J_j$, we have 
$\bigcap_{k=1}^2 I_k\neq 0$, and   \(J \cap I_k \neq 0\) for each \(k\), but 
\[	J \cap
 \left( \bigcap_{k=1}^2 I_k \right) =0. \]
So, $J$ is not an $\m$-ideal.   On the other hand,
we know simple ideals are $\m$-ideals.  This proves (2).
\end{proof}

The conditions above are not equivalent to: every proper \( \m \)-ideal of \( R \) is a simple ideal of \( R \).   For example, let $R$ be the unitization of  the finitely supported 
countably infinite matrix ring.
This has the latter property (it has one proper $\m$-ideal, but there are no proper ideals that have a complemented ideal).    

\begin{corollary}\label{nmi} A nonzero  unital ring with no nonzero  proper $\m$-ideals is simple. 
\end{corollary}

\begin{proof}  By the lemma, every proper ideal is a direct summand.
So by the theorem, $R$ is a direct sum of simple rings. Since simple ideals are $\m$-ideals, $R$ must in fact be simple. 
\end{proof}

\begin{theorem}\label{morita}
Morita equivalence of unital algebras preserves essential ideals and  $\m$-ideals.
\end{theorem}

\begin{proof}  Suppose that $(R,S,X,Y)$ is a Morita context with $X \otimes_R Y \cong S$ and $Y \otimes_S X \cong R$. 
In particular, we may write $1_R = \sum_{k=1}^n (y_k,x_k)$ and $1_S = \sum_{k=1}^m [w_k,z_k]$, with  $x_k, w_k \in X,$ and $y_k, z_k \in Y$.  The map from ideals $J$ in $S$ to ideals in $R$, for example,
takes $\{ a \in J \}$ to \[T(J) = \left\{ \sum_{k=1}^N (r_k,
a_k  s_k) \mid  N \in \bN, 
a_k \in J, r_k \in Y, s_k \in X \right\}.\] 
It is well known that $T(J \cap K) = T(J) \cap T(K)$; indeed, Morita equivalence induces a lattice isomorphism of the ideal lattice of $S$ onto the ideal lattice of $R$.  
Moreover,  
$J \cap K = (0)$ if and only if $T(J) \cap T(K) = (0)$. 
Thus, $J$ is essential if and only if $T(J)$ is essential.   Similarly, $J$ is an $\m$-ideal if and only if $T(J)$ is an $\m$-ideal.
\end{proof}

We remark that the same result and proof holds for unital rings.

The following result is well-known (see, e.g., Proposition 5.6 in \cite{GW04}) to hold for essential submodules. 

\begin{proposition}\label{krs}
Suppose that $R$ is a ring (not necessarily unital).
\begin{enumerate}
\item If $N$ and $N'$ are ideals of $R$ such that  $N$ is an $\m$-ideal of $N'$ and $N'$ is an $\m$-ideal of $R$, then $N$ is an $\m$-ideal of $R$.
 
\item Suppose that $R$ is a regular ring. Let \( A_1, A_2, B_1, B_2 \) be ideals of  \( R \). If \( A_1\) is an $\m$-ideal of \(B_1 \) and \( A_2\) is an $\m$-ideal of \(B_2 \), then \(A_1 \cap A_2\) is an $\m$-ideal of \(B_1 \cap B_2\).
\end{enumerate}
\end{proposition}

\begin{proof}
(1) 
Suppose that $ N_1$ and $ N_2$ are nonzero ideals of \( R\) such that 
$
N_1\cap N_2 \neq 0.
$
We assume that \( N \cap N_k \neq 0 \) for each \( k  \). Then $N'\cap N_k\neq 0$ for each $k$.
Since \( N'\) is an $\m$-ideal of $R$, we have \(N'\cap \left(  N_1\cap N_2\right)  \neq 0\). This implies that $0\neq  N'\cap N_k\subseteq N'$ for each $k$ with $\bigcap_{k=1}^2(N'\cap N_k)\neq 0.$ Since \( N\) is an $\m$-ideal of $N'$, moreover, we have
\[
0\neq N \cap \left(\bigcap_{k=1}^2 (N'\cap N_k)\right) =N \cap \left( N_1\cap N_2\right),
\]
which proves the claim.

(2) 
Suppose that \( A_1\) is an $\m$-ideal of $B_1$ and \( A_2 \) is an $\m$-ideal of $B_2$.  We may assume that $A_1 \cap A_2 \neq 0$. Let \(N_1\) and \(N_2\) be nonzero ideals of \( B_1 \cap B_2 \) such that \(N_1\cap N_2 \neq 0\). Suppose that \( (A_1 \cap A_2) \cap N_k \neq 0 \) for each \(k\).   
Our hypothesis implies that $N_k$ is an ideal in $R$ and in $B_1$, and hence
an ideal in both \(B_1\) and \(B_2\).
Also, $$A_1 \cap A_2 \cap (N_1\cap N_2) =  A_1 \cap (A_2 \cap N_1) \cap N_2 \neq 0,$$ since $A_1$ is an $\m$-ideal in $B_1$.
\end{proof}

\begin{remark}
We remark that the hypothesis that $R$ is regular was used to obtain that an ideal of an ideal of $R$ is an ideal of $R$.  Thus, it is unnecessary in rings with this latter property, such as $\mathds{Z}_n$.  One may also replace this hypothesis with a `relative $\m$-ideal' condition, in which 
the ideals $N_k$ are assumed to be ideals in the larger ring.  We will not take the time to develop this. 
\end{remark}

Motivated by essential complements of submodules, we shall now introduce a similar notion for $\m$-ideals and study some of their properties.

\begin{definition}
An \emph{$\m$-complement} of an ideal $N$ of a ring $R$ is an ideal $N'$ of $R$ such that $N\cap N'=0$ and $N+N'$ is an $\m$-ideal of $R$.    
\end{definition}

\begin{remark}
Note that $\m$-complements are not unique. Indeed, in $\mathds{Z}_{12}$, the ideals  $(3)$ and $(4)$ are complements of each other, as $(3)\cap (4) =(0)$ and $(3)+ (4)=\mathds{Z}_{12}$. Consider $(2)$,  $(3)\cap (2)=(6)$, and $(4)\cap (2)=(4)$. It is easily seen that $(2)$ is an $\m$-ideal. Moreover, $(3)\cap (4) = (0)$ and $(6)+(4)=(2) \mathds{Z}_{12}$. Therefore, both $(3)$ and $(6)$ are $\m$-complements of $(4)$.    
\end{remark}

\begin{proposition}
If $A$ and $N$ are ideals of a regular arithmetic ring\footnote{Recall that a ring is said to be \emph{arithmetic} if its lattice of all ideals is distributive.} $R$ such that $A\subseteq N$, and $B$ is an $\m$-complement of $A$ in $R$, then $B \cap N$ is an $\m$-complement of $A$ in $N$.
\end{proposition}

\begin{proof}
Since $A \cap B = 0$ and $B \cap N \subseteq B$, it follows that
\(
A \cap (B \cap N) \subseteq A \cap B = 0.
\) To show  $A + (B \cap N)$ is an $\m$-ideal of $N$,
let $K_1$ and $K_2$  be  nonzero ideals of $N$, and hence of $R$, such that
\(
K_1\cap K_2 \neq 0\) and \((A + (B \cap N)) \cap K_i \neq 0 \), for each $i$.
Since $A + (B \cap N) \subseteq A + B$ and $K_i \subseteq N \subseteq R$, by applying distributivity twice, we have
\[
0\neq (A + (B \cap N)) \cap K_i = (A + B) \cap K_i,\] for each $i$.
Thus, we can rewrite
\[
(A + (B \cap N)) \cap \left( K_1\cap K_2\right) = (A + B) \cap \left( K_1\cap K_2 \right).
\]
Since    
\(
K_1\cap K_2 \neq 0\) and \((A + B) \cap K_i \neq 0
\) for each $i$, and $A + B$ is an $\m$-ideal, we obtain
\(   
(A + B) \cap \left( K_1\cap K_2 \right) \neq 0.
\)    
Therefore, 
\[
(A + (B \cap N)) \cap \left( K_1\cap K_2 \right) \neq 0.
\qedhere	\]
\end{proof}

\begin{theorem}\label{mct}
If \(N\) and \(Q\) are ideals of a ring \(R\) with \(N \cap Q = 0\), then \(N\) has an \( \m \)-complement containing \(Q\).
\end{theorem} 

\begin{proof}
Consider the set \[\mathcal{S}:= \left\{ N'\mid N'\;\text{is an ideal of}\; R,\, N \cap N' = 0,\;\text{and}\; Q \subseteq N'\right\}.\]
Clearly \(Q \in  \mathcal{S}\), so \(\mathcal{S}\) is nonempty. By Zorn's Lemma, \(\mathcal{S}\) has a maximal element, say \(N'\).
We claim that \(N + N'\) is an $\m$-ideal of $R$. If not, then there exist nonzero ideals \(N_1\) and \(N_2\) of \(R\) such that \(N_1\cap N_2\neq 0\) and \((N + N') \cap N_k \neq 0\) for each \(k\), but \[(N + N') \cap \left(N_1\cap N_2 \right) = 0.\] Then by Lemma 6.3 from 
\cite{Cal00}, we have $N\cap \left(N'+(N_1\cap N_2) \right) = 0,$  contradicting the maximality of \(N'\). 
\end{proof}

The following corollary is worth comparing with Theorem 2.4.8 in \cite{Row88}.

\begin{corollary}\label{lco}
In a ring $R$, 
every ideal
$N$ of  $R$ has an 
\( \m \)-complement. 
\end{corollary}

\begin{proof} 
Take $Q=0$ in Theorem \ref{mct}. 
\end{proof}

We can also see from Theorem \ref{mct} that 
if a ring $R$ does not have any proper $\m$-ideals, then $R$ is complemented. (Indeed: if $R$ does not have any proper $\m$-ideals, then by Corollary \ref{lco}, for every ideal $N$ of $R$, there exists an ideal $N'$ of $R$ such that $R=N\oplus N'$, and hence, $R$ is complemented.
Indeed $N'$ may be taken to be  maximal among the ideals disjoint from $N$.
 In the unital case this follows from the stronger result,
Corollary \ref{nmi}.

\begin{cremark} The theory of M-ideals, as developed here, constitutes a generalization of essential ideals and provides a new perspective on structural properties of rings.
Given the important connection between `essential' and `injective' and `injective hull' in both ring theory and functional analysis (see e.g.\ \cite[Chapter 4]{BLM}), and in particular for Banach and operator space $M$-ideals (see e.g.\ Section 4.8 of \cite{BLM} and \cite{HWW93}), it is important to explore this connection both for rings and modules. It may be also fruitful to define invariants associated to $\m$-ideals, such as an ``$\m$-dimension'' of a ring, defined by the maximal length of a strictly ascending chain of $\m$-ideals, or the ``$\m$-length'' of a module, measuring the number of nontrivial $\m$-submodules in a composition series. These may provide new quantitative tools for measuring how far a ring is from being simple. 
There is an effort made in connecting essential ideals with radical theory (see \cite{OJ81, JR82, J85}). It would be interesting to see how far these ideas can be extended in the setting of  $\m$-ideals.

Characterizations of essential ideals for rings of measurable functions have been done in \cite{Mom10}, and similar results have been obtained for quotient Pr\"{u}fer domains and structural matrix rings in \cite{GW89}. It may be worth seeing how the characterizations of $\m$-ideals in these rings differ from essential ideals.

Given the topological motivations behind $\m$-ideals in Banach spaces, it is reasonable to investigate their analogues in topological or linearly topological rings. In particular, one may study the behavior of $\m$-ideals under completions, localizations, or limits. It would be worthwhile to determine whether $\m$-ideals are preserved under flat or faithfully flat base change and how they interact with completions in the topological sense, especially in rings of continuous functions or formal power series.
\end{cremark}

\subsection*{Acknowledgements}D.P. Blecher acknowledges support from NSF Grant DMS-2154903, and also from the School of Mathematical and Statistical Sciences at NWU Potchefstroom, which helped make an international visit possible.  We also thank Francois Schulz for introducing us to each other.  Finally we thank the referee and the editor for their very kind assistance.  AI was only used after acceptance of the paper, and minimally, to correct a few typos, unclarities, language, etc.


\begin{thebibliography}{99}


\bibitem{AE72}
E. M. Alfsen and E. G. Effros, Structure in real Banach spaces. Part I, \textit{Ann. Math.}, \textbf{96}(1) (1972), 98--128.

\bibitem{Arm68}
E. P. Armendariz, Direct and subdirect sums of simple rings with unit, \textit{Amer. Math. Monthly}, \textbf{75}
 (1968), 746--748.
 
\bibitem{Bla06} B. Blackadar, \textit{Operator algebras. Theory of $C^*\!\!$-algebras and von Neumann algebras,} Encyclopaedia of Mathematical Sciences, 122. Operator Algebras and Non-commutative Geometry, III. Springer-Verlag, Berlin, 2006. 

\bibitem{BLM} D. P. Blecher and C.  Le Merdy, \textit{Operator algebras and their modules---an
operator space approach,} Oxford Univ.\  Press, Oxford (2004).

\bibitem{Cal00}
G. C\v{a}lug\v{a}reanu, \textit{Lattice concepts of module theory}, Kluwer Academic Publishers, Dordrecht, 2000.
\bibitem{GCM25} E. Caviglia, A. Goswami, and L. Mesiti, $\mu$-elements: an extension of essential elements, \textit{Algebra Universalis}, \textbf{87} (2026),  Paper No.\ 8,
21 pp.

\bibitem{GJ60} L. Gillman and M. Jerison, \textit{Rings of continuous functions}, D. Van Nostrand Company, Inc., 1960.


\bibitem{GW04} K. R. Goodearl and R. B. Warfield Jr., \textit{An introduction to noncommutative
Noetherian rings}, 2nd edition, Cambridge University Press, 2004.

\bibitem{GW89}
B. W. Green and L. van Wyk, On the small and essential ideals in certain classes of rings, \textit{J. Aust. Math Soc. (Series A)}, \textbf{46} (1989), 262--271.

\bibitem{HWW93} P. Harmand, D. Werner, and W. Werner,  \textit{$M$-ideals in Banach spaces and Banach algebras,} Springer, 1993.

\bibitem{J85}
T. L. Jenkins, Essential ideals and radical classes in \textit{Radical theory} (Eger, 1982), 169--181.
Colloq. Math. Soc. János Bolyai, 38
North-Holland Publishing Co., Amsterdam, 1985.

\bibitem{JR82}
T. L. Jenkins and H. J. le Roux, 
Essential ideals, homomorphically closed classes and their radicals,
\textit{J. Austral. Math. Soc. Ser. A}, \textbf{33}(3) (1982),  356--363.

\bibitem{Lam01} T. Y. Lam, \textit{A first course in noncommutative rings}, 2nd ed., Springer, 2001.
 
 \bibitem{OJ81}
 D. M. Olson and T. L. Jenkins, Upper radicals and essential ideals, \textit{J. Aust. Math. Soc.}, \textbf{30} (1981), 385--389. 

 \bibitem{Mom10}
 E. Momtahan, Essential ideals in rings of measurable functions, \textit{Comm. Alg.}, \textbf{38}(12) (2010), 4739--4746.

 \bibitem{Row88} L. H. Rowen, \textit{Ring theory}, vol. 1, Academic Press Inc., 1988.
\end{thebibliography}
\end{document}